\documentclass[11pt]{amsart}
\usepackage[latin1]{inputenc}
\usepackage{graphicx}
\usepackage{amsmath}
\usepackage{amssymb}

\usepackage{amsmath,amsthm,amsfonts,amssymb,amscd}
\usepackage{enumerate}% http://ctan.org/pkg/enumerate
\usepackage{amssymb,amsthm,amsmath,psfrag,mathrsfs}
\usepackage[latin1]{inputenc}
\usepackage{graphicx}
\usepackage{comment}

\newcommand{\Z}{\mathbb{Z}}
\newcommand{\C}{\mathbb{C}}

\newcommand{\menos}{\backslash}

\def\Z{\mathbb{Z}}                 %Integer  numbers
\def\C{\mathbb{C}}                   %Complex numbers
\def\fa{{\mathcal{F}}}

\def\po{{\partial}}
\def\ro{{\rho}}

\def\vr{{\varphi}}

\def\la{{\lambda}}

\def\ov{\overline}

\def\bz{{\mathbb{Z}}}

\def\bc{{\mathbb{C}}}

\def\sing{\operatorname{{sing}}}

\def\Sat{\operatorname{{Sat}}}

\def\leaderfill{\leaders\hbox to .8em{\hss .\hss}\hfill}
\def\_#1{{\lower 0.7ex\hbox{}}_{#1}}

\def\fa{{\mathcal{F}}}

\def\po{{\partial}}
\def\ro{{\rho}}

\def\vr{{\varphi}}

\def\la{{\lambda}}

\def\ov{\overline}

\def\bz{{\mathbb{Z}}}

\def\bc{{\mathbb{C}}}

\def\sing{\operatorname{{sing}}}

\def\Sat{\operatorname{{Sat}}}

\renewcommand{\thefootnote}

\newtheorem{theorem}{Theorem}

\newtheorem{Theorem}{Theorem}
\newtheorem{Lemma}{Lemma}

\newtheorem{obs}{Remark}

\newtheorem{defi}{Definition}

\newtheorem{Claim}{Claim}
\theoremstyle{Definition}

\theoremstyle{Remark}

\numberwithin{equation}{section}

\title[On holomorphic $\C^*$-actions]{On holomorphic $\C^*$-actions}
\author{V. Le\'on and B. Sc\'ardua}
\address{V. Le\'on. ILACVN - CICN, Universidade Federal da Integração Latino-Americana, Parque tecnológico de Itaipu, Foz do Iguaçu-PR, 85867-970 - Brazil}
\email{victor.leon@unila.edu.br}
\address{B. Sc\'ardua. Instituto de Matem\'atica - Universidade Federal do Rio de Janeiro,
CP. 68530-Rio de Janeiro-RJ, 21945-970 - Brazil}
\email{bruno.scardua@gmail.com}

\date{}

\begin{document}

\maketitle

\begin{abstract} In this paper we study holomorphic actions of the complex multiplicative group on
complex manifolds around a singular (fixed) point. We prove linearization results for the germ of action
and also for the whole action under some conditions on the manifold. This can be seen as a follow-up to previous works of M. Suzuki and other authors.
\end{abstract}
%\tableofcontents

\section{Introduction} \label{Section:intro}

A fundamental contribution to the study of holomorphic actions  on  Stein surfaces, was made by M. Suzuki who
introduced  the use of techniques of Theory of Foliations, Potential
Theory and the Theory of Analytic Spaces (cf. \cite{Suzuki1} and
\cite{Suzuki2}).  Given a holomorphic action $\vr$ of the complex multiplicative group $\bc^*$ on a complex analytic space $N^2$ of dimension two,  we denote by $\fa_\vr$ the holomorphic foliation on $N^2$ induced by $\vr$. In \cite{Suzuki1} it is proved that if $V$ is Stein then the  $\fa_\vr$ admits a meromorphic first integral. In a very sharp study (\cite{Suzuki2}) the same author proves that any  $\bc^*$-action on the affine space $\bc^2$ is analytically linearizable, {\it i.e.,} conjugate to an action $s \circ (x,y)=(s^n x, s^m y)$ where $s \in \bc^*, \, (x,y) \in \bc^2$, and $m,n \in \mathbb Z$.

As a sort of extension of the above result, in \cite{C-M-S} the authors consider the classification  of a Stein complex analytic space $V$ of dimension two, with a normal singularity $p\in V$,  endowed with a  $\bc^*$-action $\vr$  having a {\it dicritical singularity} at $p$, meaning that $p$  is a fixed point  such that   every non-singular orbit of $\vr$, close enough to $p$  accumulates at and  only at $p$. These are called {\it quasi-homogeneous singularities} in the framework of Analytic and Algebraic Geometry (\cite{Seade}).

In \cite{Camacho-Scardua} the authors study    a $\bc^*$-action $\vr$ with isolated singularities  on a nonsingular complex analytic Stein space $V$, the action having some nondicritical singularity and   without dicritical singularities. In few words it is proved analytic linearization at the level of leaf spaces.

\subsection{Suzuki's theory}
\label{section:suzukitheory}

For our current purposes Suzuki's main result probably is:

\begin{Theorem}[\cite{Suzuki2}]
Any analytic $\bc^*$-action on $\bc^2$ is analytically linearizable, {\it i.e.}, analytically equivalent to an operation of the form $s\circ (x,y)=(s^n x, s^m y), \, s \in \bc^*, (x,y)\in \bc^2$, for
some $n,m \in \mathbb N$.
\end{Theorem}

 The classification of holomorphic $\bc$-actions (i.e., holomorphic flows) with proper orbits on $\bc^2$ is also obtained by Suzuki (\cite{Suzuki2}, Theorem 4), and not all of them are linearizable.

\subsection{Quasi-homogeneous singularities: Orlik-Wagreich results}

 A 2-dimensional complex analytic variety $V$, with a distinguished
point $p\in V$,  is called a
{\it quasi-homogeneous} complex surface singularity,
if it admits a holomorphic action of the complex multiplicative group
$\bc^*=\bc \backslash \{0\}$ such that every non-singular orbit in a neighborhood of $p$, accumulates only at
$p\in V$ (see for instance \cite{Seade}, page 67 and \cite{CamachoScardua2, Scardua}).
Such an action is called a {\it good action}.
The study of algebraic quasi-homogeneous singularities is a main
topic in the theory of singularities with major contributions from Saito (\cite{saito71}), giving  an algebraic description in the local context, and from Orlik and Wagreich (\cite{Orlik}, \cite{Orlik2}). The later  studied algebraic affine
 varieties of dimension two,  with an isolated
singularity at the origin, under the hypothesis that the variety is invariant by an  algebraic action
of the form \,
$\sigma_Q(t,(z_{0},...,z_{n}))=(t^{q_{0}}z_{0},...,
t^{q_{n}}z_{n})$ where $Q=(q_0,...,q_n) \in\mathbb N^{n+1}$, i.e.
all $q_{i}$ are positive integers. Algebraic surfaces embedded in
$\mathbb{C}^{3}$ endowed with such an action are completely classified by them.

\subsection{Main results}
In this paper we study linearization of $\bc^*$-actions and apply our results to some of the previous works above.
We denote by $(\varphi,G,M,p)$ a holomorphic  action $\varphi$ of a complex Lie group $G$ on a complex manifold $M$ having $p \in M$ as fixed point (not necessarily isolated). Given an element $g\in G$ we shall denote by
 $\varphi ^g \colon M \to M$ the (biholomorphic) map $M \ni x\mapsto \varphi (g,x)$. Let $(\varphi,G,M,p)$ be  a holomorphic  action, with a fixed point at $p\in M$ and let $\xi: W \ni p \to \C^n$ be a local chart. Using the natural identification $T^{(1,0)}\C^n \cong \C^n$, we can use the derivative of the $\varphi$ at $p$ to define the following holomorphic action:
\begin{eqnarray*} \psi=\psi_{p,\xi}: G \times \C^n &\to& \C^n \\ (g,x) &\mapsto& \left[D \xi(p) \cdot D \varphi ^g(p) \cdot  \left(D \xi (p) \right)^{-1} \right] \cdot x.
\end{eqnarray*}
Although we have used $\xi$ to define $\psi$, any other chart defines an equivalent action. So we will refer to $\psi$ as {\it the action by linear transformations} defined by the derivative of $\varphi$ at $p$ and may  write $\psi=(D\varphi)_p$.

In the first part of this work, we study the case $M=\C^n$, $\xi \equiv Id$ and we obtain the following semilocal linearization result below:

\vglue.1in

\begin{theorem} [Semi-local linearization in the affine space]
\label{theorem:semilocallin}
Let $(\varphi,\C^*,\C^n,p)$ be a holomorphic  action of $\bc^*$ on $\bc^n$ having  $p\in \bc^n$ as a fixed point and let $(\psi, \C^*,\C^n,0)$ be the action $\psi=(D\varphi)_p$  by linear transformations  defined by the derivative of $\varphi$ at $p$. Then there is an entire   map $F\colon  \C^n \to \C^n$ such that $F(p)=0$ and
$$\psi^z \circ F \equiv F \circ \varphi^z$$
for each $z \in \C^*$. Moreover, there is a $\varphi$-invariant  open set $W \ni p$ in $\C^n$ such that $F|_W:W \to F(W)$ is a biholomorphism.
\end{theorem}

Next we study actions on complex manifolds. Let $M$ be a complex manifold  of dimension $n$  and $\vr:\mathbb{C^*}\times M\to M$ a holomorphic action of the group $\mathbb{C^*}$ on $M$. Denote by $\fa_\vr$ the
one-dimensional holomorphic foliation with isolated singularities on $M$ induced by $\vr$. The one-dimensional orbits of $\vr$ are {\it leaves} of $\fa_\vr$ and the {\it singular set} of  $\sing(\fa_\vr)$ is contained in the set of  fixed points of $\vr$. Following the foliation theory terminology, an isolated singularity $p$ of  a one-dimension holomorphic foliation $\fa$ in dimension two will be called  {\it dicritical} if, for some neighborhood $W$ of $p$, there are infinitely many leaves of the restriction $\mathcal {F}\big|_{W}$ accumulating {\em only at} $p$. The closure $\ov{L_W}=L_W \cup \{p\}\subset W$  of any such a local leaf $L_W \subset W$ is an invariant local analytic
curve in $W$,  called a {\it separatrix} of $\mathcal{F}$ through $p$. Thus,
a dicritical singularity exhibits infinitely many separatrices. A singularity $p\in V$ of a good $\mathbb C^*$-action on $V$ is clearly  dicritical.  With a notion of dicricity similar to the one above we obtain the following linearization result:

\vglue.1in
\begin{theorem}[Local linearization on a manifold]
  \label{Theorem:B}\, Let $(\varphi,\C^*,M,p)$ be a holomorphic  action on a complex manifold $M$, with an isolated fixed point at $p\in M$ such that $p$ is of the dicritical type. Let  $(\psi, \C^*,\C^n,0)$ be the action by linear transformations defined by the derivative of $\psi=(D\varphi)_p$ of $\varphi$ at $p$. There are a $\varphi$-invariant open set $A\ni p$ in $M$, a $\psi$-invariant open set $B\ni 0$ in $\C^n$ and a biholomorphism
$T\colon  A \to B$ such that
$$\psi^z \circ T  \equiv T \circ \varphi^z$$
for each $z \in \C^*$.\end{theorem}
In the last section we apply our techniques to the classification of $\bc^*$-actions on Stein manifolds, in particular we obtain variants of Suzuki's results above (cf. Theorems~\ref{Theorem:dicriticalStein} and ~\ref{Theorem:surface}).

\noindent{\bf Acknowledgement}. This work was partially supported by Faperj and CNPq. On behalf of all authors, the corresponding author states that there is no conflict of interest. After submitting this manuscript we were informed that 
some of our results are connected with Theorem 1.4 in \cite{K-S}. We are grateful to the authors for letting us know that.
Our techniques are quite analytical and also apply to non-Stein manifolds, as one can check in Theorem~\ref{Theorem:B}.

\section{Local linearization of $\C^*$-actions}

In this section we prove Theorem~A. We shall write $\psi_p$ to denote the action $(D\varphi)_p$. We begin with the following lemma, a sort of local version of Bochner-Cartan theorem for our situation:

\begin{Lemma}\label{bochner}Let $(\varphi,\C^*,\C^n,p)$ be a holomorphic  action with a fixed point at $p$ and let $(\psi_{p}, \C^*,\C^n,0)$ be the action by linear transformations defined by the derivative of $\varphi$ at $p$. There is a holomorphic  map $F_p: \C^n \to \C^n$ such that:
\begin{itemize}\item[$i.$]$\psi_p^z \circ F_p \equiv F_p \circ \varphi^z, \ \forall z \in \C^*$.
               \item[$ii.$]$F_p(p)=0$ e $DF_p(p)=Id$.
\end{itemize}
\end{Lemma}

%%%%%%%%%%%%%%%%%%%%%%%%%%%%%%%%%%%%%%%%
%%%%%%%%%%%%%%%%%%%%%%%%%%%%%%%%%%%%%%%%
%%%%%%%%%%%%%%%%%%%%%%%%%%%%%%%%%%%%%%%%
\begin{proof}[Proof]Consider the action (holomorphic flow) on $\mathbb C^n$
\begin{eqnarray*} \phi: \C \times \C^n &\to& \C^n \\  (z,x) &\to& \varphi (\exp(2 \pi \sqrt{-1}z), x).
\end{eqnarray*}
We define the holomorphic map
\begin{eqnarray*} F_p: \C^n &\to& \C^n \\ x &\mapsto& \int_0^1 D \phi ^{-s}(p) \cdot \left[\phi ^s(x) - p\right]ds.
\end{eqnarray*}
We write $F_p$ in the form
$$F_p(x) = \int_0^1 D \phi ^{-s}(p) \cdot \phi ^s(x)ds - c(p),$$
where $c(p)= \int_0^1 D \phi ^{-s}(p) \cdot p \ ds$. If $0 \leq  t \leq 1$, then
\begin{eqnarray}\label{equa1}\nonumber D \phi ^t(p) \cdot \int_0^1 D \phi ^{-s}(p) \cdot \phi ^s(x)ds &=& \int_0^1 D\phi ^{t-s}(p) \cdot \phi ^s(x)ds \\ &=& \int_{-t}^{1-t}D \phi ^{-u}(p) \cdot \phi ^{u+t}(x)du \nonumber \\ &=& \int_{-t}^0\cdots + \int_0^{1-t}\cdots. \end{eqnarray}
Since $\varphi ^1 \equiv Id$, we have that
\begin{eqnarray}\label{equa2}\nonumber  \int_{-t}^0 D \phi ^{-u}(p) \cdot \phi ^{u+t}(x) du &=& \int_{-t}^0 D \phi ^{-u}(p) \cdot D \phi ^{-1}(p) \cdot \phi ^1 \circ \phi ^{u+t}(x) du \\&=& \int_{-t}^0 D \phi ^{-u -1} (p) \cdot \phi ^{u+1+t}(x) du \nonumber \\ &=& \int_{1-t}^1 D\phi ^{- \alpha}(p) \cdot \phi ^{\alpha + t}(x) d\alpha. \end{eqnarray}
It follows from ($\ref{equa1}$) and ($\ref{equa2}$) that
\begin{equation}\label{equa3}D \phi ^t(p) \cdot \int_0^1 D \phi ^{-s}(p) \cdot \phi ^s(x)ds = \int_0^1 D \phi ^{-u}(p) \cdot \phi ^{u+t}(x) du.\end{equation}
A similar computation shows that $D \phi ^t(p)\cdot c(p)=c(p)$. This, together with ($\ref{equa3}$) implies that
$$D \phi ^t(p) \cdot F_p \equiv F_p \circ \phi ^t,$$
for all $t \in [0,1]$. Now, fix an arbitrary $x \in
\C^n$. The map
\begin{eqnarray*}Q_x: \C &\to& \C^m \\ z &\mapsto& \left[D\phi ^z(p)\cdot F_p- F_p \circ \phi ^z\right](x) \end{eqnarray*}
is holomorphic. Since $Q_x|_{[0,1]} \equiv 0$, we have that $Q_x$ is identically zero. Since $x$  was taken arbitrarily,
the proof is finished.
\end{proof}

%%%%%%%%%%%%%%%%%%%%%%%%%%%%%%%%%%%%%%%%
%%%%%%%%%%%%%%%%%%%%%%%%%%%%%%%%%%%%%%%%
%%%%%%%%%%%%%%%%%%%%%%%%%%%%%%%%%%%%%%%%

\begin{Lemma}\label{lema2}Let $(\varphi,\C^*,\C^n,p)$ be a holomorphic  action with a fixed point at $p$ and let $(\psi_{p}, \C^*,\C^n,0)$ be the action by linear transformations defined by the derivative of $\varphi$ at $p$. Let $F_p: \C^n \to \C^n$ be the holomorphic map defined in Lemma $\ref{bochner}$. If $U \ni p$ and $V=F_p(U)$ are open sets in $\C^n$ satisfying
\begin{enumerate}\item $F_p$ is injective in $U$ and
                 \item $\psi_p^z(V)\cap V$ is connected, for each $z \in \C^*$,
\end{enumerate}
then $F_p$ is injective in the saturation  ${\rm Sat}(U,\varphi)$ of $U$ by (the orbits of) $\varphi$.
\end{Lemma}

%%%%%%%%%%%%%%%%%%%%%%%%%%%%%%%%%%%%%%%%
%%%%%%%%%%%%%%%%%%%%%%%%%%%%%%%%%%%%%%%%
%%%%%%%%%%%%%%%%%%%%%%%%%%%%%%%%%%%%%%%%
\begin{proof}[Proof]
The saturation  ${\rm Sat}(U,\varphi)$ of $U$ by $\varphi$ is the union of all orbits of $\varphi$ intersecting $U$.
For each $z \in \C^*$, it follows from Lemma $\ref{bochner}$ that
\begin{equation}\label{equ}F_p^{-1}\circ \psi_p^z \equiv \varphi^z\circ F_p^{-1}\end{equation}
in some open set $W_z \subset \psi_p^{-z}(V)\cap V$. Using $(2)$ and the Identity Principle, we conclude that ($\ref{equ}$) happens in $\psi_p^{-z}(V)\cap V$. We define the map
\begin{eqnarray*}T: {\rm Sat}(V,\psi_p) &\to& \Sat(U,\varphi) \\ y &\mapsto& \varphi^{z^{-1}} \circ F_p^{-1} \circ \psi_p^z(y),
\end{eqnarray*}
where $z\in \C^*$ satisfies $\psi_p^z(y) \in V$. We need to show
that $T$ is well defined. It is obvious if $y \in {\rm
Fix}(\psi_p)$. Take $y \in  {\rm Sat}(V,\psi_p) \menos {\rm
Fix}(\psi_p)$. Suppose that there exist $\tilde{z} \ne z$ such
that $\psi_p^z(y) =x \in V$ and $\psi_p^{\tilde{z}}(y) = \tilde{x}
\in V$. In this case, there is a $l \in \mathbb{C}^*$ such that
$\psi_p^l(x)=\tilde{x}$ and this implies that
$$\psi_p^{l^{-1}} \circ \psi_p^{\tilde{z}} \circ \psi_p^{z^{-1}}(x)=x,$$
that is, $l^{-1} * \tilde{z} * z^{-1} =h \in \C^*_{x,\psi_p}$.
Since $x \in \psi_p^{h^{-1}}(V) \cap V$, it follows from
($\ref{equ}$) that $h \in \C^*_{F_p^{-1}(x),\varphi}$. Therefore,
\begin{eqnarray*}\varphi^{\tilde{z}^{-1}} \circ F_p^{-1} \circ \psi_p^{\tilde{z}}(y) &=& \varphi^{z^{-1}} \circ \varphi^{h^{-1}} \circ \varphi^{l^{-1}}
\circ F_p^{-1} \circ \psi_p^l \circ \psi_p^h \circ \psi_p^z(y) \\ &=&  \varphi^{z^{-1}} \circ \varphi^{h^{-1}} \circ \varphi^{l^{-1}}
\circ F_p^{-1} \circ \overbrace{\psi_p^l \circ  \underbrace{\psi_p^z (y)}_{\in V}}^{\in V} \\  &=& \varphi^{z^{-1}} \circ \varphi^{h^{-1}} \circ  F_p^{-1} \circ  \psi_p^z (y) \\ &=& \varphi^{-z} \circ F_p^{-1} \circ  \psi_p^z  (y).
\end{eqnarray*} This implies that $T$ is well defined and that is holomorphic, since $V$ is open. We define the map
\begin{eqnarray*}L: \Sat(U,\varphi) &\to& \Sat(V,\psi_p) \\ y &\mapsto& \psi_p^{z^{-1}} \circ F_p \circ \varphi^z(y),
\end{eqnarray*}
where $z \in \C^*$ satisfies $\varphi^z(y) \in U$. By item ($i$) of Lemma $\ref{bochner}$ we have that $L \equiv F_p$. On the other hand $T \circ L \equiv Id|_{{\rm Sat}(U,\varphi)}$ and $L \circ T \equiv Id|_{{\rm Sat}(V,\psi_p)}$, that is, $T$ is the inverse of $F_p$.
\end{proof}
%%%%%%%%%%%%%%%%%%%%%%%%%%%%%%%%%%%%%%%%
%%%%%%%%%%%%%%%%%%%%%%%%%%%%%%%%%%%%%%%%
%%%%%%%%%%%%%%%%%%%%%%%%%%%%%%%%%%%%%%%%

Now, we can prove Theorem~A:

%%%%%%%%%%%%%%%%%%%%%%%%%%%%%%%%%%%%%%%%
%%%%%%%%%%%%%%%%%%%%%%%%%%%%%%%%%%%%%%%%
%%%%%%%%%%%%%%%%%%%%%%%%%%%%%%%%%%%%%%%%
\begin{proof}[Proof of Theorem~A]Consider the flow action
\begin{eqnarray*} \phi: \C \times \C^n &\to& \C^n \\  (z,x) &\to& \varphi (\exp(2 \pi \sqrt{-1}z), x).
\end{eqnarray*} Let $X(x)= \sum_{j=1}^n a_j(x) \frac{\partial}{\partial z_j} $ be the complete holomorphic vector field whose flow is $\phi$. If we identify $X$ with the map $X(x)=(a_1(x), \dots,a_n(x))$, then the action
\begin{eqnarray*} \psi: \C \times \C^n &\to& \C^n \\  (z,x) &\to& \psi_p (\exp(2 \pi \sqrt{-1}z), x)\end{eqnarray*}
is the flow of the linear vector field $Y(x)=DX(p)\cdot x$. Since the isotropy of $\psi_p$ at each point contains $\Z$, a Jordan canonical form for $DX(p)$ has no nilpotent part. Therefore, there is a linear isomorphism $A$ and an action $\phi: \C \times \C^n \to \C^n$ such that
$$\phi^z(x) = A^{-1} \circ \psi^z \circ A(x) = \left(\exp(\lambda_1 2 \pi \sqrt{-1} z) x_1, \dots , \exp(\lambda_n 2 \pi \sqrt{-1} z) x_n\right),$$
where $x=(x_1, \dots, x_n)$ and $\lambda_i \in \Z$. Let $B_0$ be an open ball centered at the origin of $\C^n$.
\begin{Claim}
\label{Claim:connected}
 For each $z \in \C$, $\phi^z(B_0) \cap B_0$ is connected.
 \end{Claim}

\begin{proof} In fact, take $z = a-bi \in \C$ and $y \in \phi^z(B_0) \cap B_0$. There exists $y_1 \in B_0$ such that $\phi^z(y_1)=y$. Let $r_{y_1}$ be the straight line segment joining $y_1$ and the origin. It is enough to show that $l_y:=\phi^z(r_{y_1}) \subset B_0$. So take $r \in r_{y_1}$. Each real coordinate of $r$ has modulus less than the modulus of respective coordinate of $y_1$. So the same is valid for their complex coordinates. Since the $j$-th complex coordinate of $\phi(z,x)$ is
\begin{equation}\label{equ1}\exp\left(\lambda_j 2 \pi b \right) \cdot \exp\left(\lambda_j 2 \pi \sqrt{-1} a \right)x_j,\end{equation}
we conclude that each complex coordinate of $\phi^z(r)$ has modulus less than the modulus of respective coordinate of $\phi^z(y_1)=y \in B_0$. Therefore, $\phi^z(r) \in B_0$ and Claim~\ref{Claim:connected} is proved.
\end{proof}

Let $F_p: \C^n \to \C^n$ be the map obtained in Lemma $\ref{bochner}$. Let $N \ni p$ be an open set in $\C^n$ such that $F_p|_N: N \to F_p(N)$
is a biholomorphism. Take an open ball $B_0$ centered at origin in $\C^n$ such that $V := A(B_0) \subset F_p(N)$. We have that
$$\psi^{z} (V) \cap V = A \left[A^{-1} \circ \psi^z \circ A(B_0) \cap B_0\right]= A \left[\phi^z(B_0) \cap B_0\right].$$
Therefore, it follows from above claim that  $\psi^{z} (V) \cap V$ is connected for each $z \in \C$. If we take $U:= (F_p|_N)^{-1}(V)$, then the desired result follows from Lemma $\ref{lema2}$.
\end{proof}
%%%%%%%%%%%%%%%%%%%%%%%%%%%%%%%%%%%%%%%%
%%%%%%%%%%%%%%%%%%%%%%%%%%%%%%%%%%%%%%%%
%%%%%%%%%%%%%%%%%%%%%%%%%%%%%%%%%%%%%%%%

\section{Actions of $\bc^*$ on complex manifolds}
Now we search for results similar to Theorem~A for complex manifolds. With this aim
the following definition is given:

\begin{defi}\label{def1}\rm{Let $(\varphi,G,M,p)$ and $(\psi,G,N,q)$ be holomorphic actions on complex manifolds $M$ and $N$. Let $U \ni p$ and $V \ni q$ connected open sets in $M$ and $N$ respectively and let $F\colon U \to V$ be a biholomorphism such that $F(p)=q$. We write
$$\varphi \sim_{pq} \psi \textnormal{ by } F\colon U \to V$$
if, for each $g \in G$, there exists an open set $W_g \subset \left(U \cap \varphi ^{g^{-1}}(U)\right)$ such that
$$\psi^g \circ F \equiv F \circ \varphi ^g \textnormal{ in } W_g.$$
Occasionally, we will use the notation $\varphi \sim_{pq} \psi$ without mentioning the map $F$.}
\end{defi}

\begin{obs}\rm{It is easy to see that $\sim$ is symmetric. More precisely, if $\varphi \sim_{pq} \psi$ by $F : U \to V$, then $\psi \sim_{qp} \varphi$ by $F^{-1}:V \to U$.}
\end{obs}

\begin{obs}\label{loc}\rm{Suppose that $\varphi \sim_{pq} \psi$ by $F\colon  U \to V$. Let $D \ni p$ be an connected open set such that $\overline{D} \subset U$. Since $p$ is a fixed point, there exists an open set $W \ni e$ (the neutral element) in $G$, such that
$\varphi ^h(D) \subset U$ when $h \in W$. So by Identity Principle, we have that, for each $h \in W$,
$$\psi^h \circ F \equiv F \circ \varphi ^h \textnormal{ in }D.$$
In addiction, it also follows from Identity Principle that if $G$ is connected, then $F\left(\mbox{Fix}(\varphi,D)\right) \subset \mbox{Fix}(\psi,F(D))$. So using the symmetry of $\sim$, we conclude that $F\left(\mbox{Fix}(\varphi,D)\right)=\mbox{Fix}(\psi,F(D))$}\end{obs}

\begin{defi}\rm{Let $(\varphi,G,M,p)$ be a holomorphic  action with a fixed point at $p$ and let $(\psi_{p}, G,\C^n,0)$ be the action by linear transformations defined by the derivative of $\varphi$ at $p$. If $\varphi \sim_{p0} \psi_p$, we say that $\varphi$ is {\it locally linearizable at $p$.}}
\end{defi}

A very preliminary result towards Theorem~B is the following:

\begin{Lemma}\label{lli}Every holomorphic action $(\varphi,\C^*,M,p)$  with a fixed point at $p$ is locally linearizable at $p$.
\end{Lemma}

%%%%%%%%%%%%%%%%%%%%%%%%%%%%%%%%%%%%%%%%
%%%%%%%%%%%%%%%%%%%%%%%%%%%%%%%%%%%%%%%%
%%%%%%%%%%%%%%%%%%%%%%%%%%%%%%%%%%%%%%%%
\begin{proof} Consider the action
\begin{eqnarray*} \phi: \C \times M &\to& M \\  (z,x) &\to& \varphi (\exp(2 \pi \sqrt{-1}z), x).
\end{eqnarray*}
Let $(\xi,W)$ be a local chart at $p$ and let $(\psi_p,\C^*,\C^n,0)$
be the action by linear transformations defined by the derivative
of $\varphi$ at $p$, constructed from $\xi$. Since $p$ is a fixed
point, there exists an open set $U \subset W$ such that $\varphi ^t(U)
\subset W$ for all $t \in [0,1]$. We define the map
\begin{eqnarray*}F=F_{p,\xi}: U &\to& \C^n \\ x &\mapsto& \int_0^1 \psi_p^{-t} \circ \xi \circ \varphi ^s(x) ds.\end{eqnarray*}
Fixed $t \in [0,1]$, we can repeat the argument in Lemma
$\ref{bochner}$ to obtain
\begin{equation}\label{igual}\psi_p^t \circ F \equiv F \circ \varphi_p^t {\rm \ em \  } \varphi ^{-t}(U)\cap U.\end{equation}
Since $DF(p)=D\xi(p)$, we can assume that $F$ is injective,
reducing $U$ if necessary. Take a connected open set $A \ni 0$ in
$\C$ and an open set $B \ni p$ in $M$ such that $\varphi ^z(x) \in U$
whenever $x \in B$, $z \in A$. Take $x \in B$ and define the
holomorphic map
\begin{eqnarray*}R_x: A \subset \C &\to \C^n \\ z &\mapsto& \left[\psi_p^z \circ F - F \circ \varphi ^z\right](x).\end{eqnarray*}
By ($\ref{igual}$), $R_x(z)=0$ whenever $z \in [0,1]\cap A$.
Therefore $R_x \equiv 0 \textnormal{ for all }x \in B$, that is,
\begin{equation}\label{igu2} \psi_p^z \circ F (x) = F \circ \varphi ^z (x) \textnormal{ whenever } x\in B, \ z \in A \ni 0.\end{equation}
The set
$$S=\{z \in \C ; \ \psi_p^z \circ F \equiv F \circ \varphi ^z \textnormal{in some neighborhood of }p \textnormal{ in }M\}$$
is a subgroup of $\C$, so by ($\ref{igu2}$), $A \subset S$. This
implies that $S=\C$ and the proof is complete.
\end{proof}

\begin{defi}\label{forte}\rm{Let
\begin{itemize}\item $(\varphi,G,M,p)$ and $(\psi,G,N,q)$ be holomorphic actions;
               \item $U \ni p$ and $V \ni q$ be open sets in $M$ and $N^2$ respectively;
               \item $D \ni p$ be an open set in $M$ such that $\overline{D} \subset U$;
               \item $F\colon U \to V$ be a biholomorphism such that $F(p)=q$.
\end{itemize}
We write
$$\varphi \simeq_{pq} \psi \textnormal{ by } F_D : U \to V$$
if, for each $g \in G$, some of the following alternatives occurs:
\begin{enumerate} \item $\varphi ^g(D) \subseteq U$ and $\psi^g \circ F \equiv F \circ \varphi ^g$ in $D$,
                  \item $\varphi ^g(U) \supseteq U$ and $\psi^g \circ F \equiv F \circ \varphi ^g$ in $\varphi ^{g^{-1}}(U)$.
\end{enumerate}
Occasionally, we will use the notation $\varphi \simeq_{pq} \psi$ without mentioning the application and the open sets.}
\end{defi}

\begin{Lemma}\label{ref} $\varphi \simeq_{pq} \psi$ implies that $\varphi \sim_{pq} \psi$. Moreover, the relation $\simeq_{pq}$ is symmetric: if $\varphi \simeq_{pq} \psi$ by $F_D : U \to V$, then $\psi \simeq_{qp} \varphi$ by $F^{-1}_E:V \to U$, where $E=F(D)$.\end{Lemma}
\begin{proof}[Proof]The proof is an easy verification.\end{proof}

\begin{Lemma}\label{comp} Let $(\varphi,G,M,p)$ and $(\psi,G,N,q)$ be holomorphic actions such that $\varphi \sim_{pq} \psi$ by $F\colon U \to V$. Suppose that there exists a connected open set $D \subset \subset U$ such that, for each $g \in G$, some of the following alternatives occurs:
\begin{enumerate} \item $\varphi ^g(D) \subseteq U$,
                  \item $\varphi ^g(U) \supseteq U$.
\end{enumerate}
Then, $\varphi \simeq_{pq} \psi$ by $F_D: U \to V$.
\end{Lemma}

\begin{proof}[Proof]This is an immediate consequence of the Identity Principle.\end{proof}

\begin{Lemma}\label{sat}If $(\varphi,G,M,p)$ and $(\psi,G,N,q)$ are holomorphic actions such that $\varphi \simeq_{pq} \psi$ by $F_D: U \to V$ then there is a biholomorphism $T: \Sat(D,\varphi) \to \Sat(F(D),\psi)$ that extends $F|_D$ and satisfies
$$\psi \circ T \equiv T \circ \varphi \textnormal{ in }\Sat(D,\varphi).$$
\end{Lemma}

%%%%%%%%%%%%%%%%%%%%%%%%%%%%%%%
%%%%%%%%%%%%%%%%%%%%%%%%%%%%%%%
%%%%%%%%%%%%%%%%%%%%%%%%%%%%%%%
\begin{proof}[Proof]We define
\begin{eqnarray*}T: \Sat(D,\varphi) &\to& \Sat(F(D),\psi) \\ y &\mapsto& \psi^{g^{-1}} \circ F \circ \varphi ^g(y),
\end{eqnarray*}
where $g \in G$ satisfies $\varphi ^g(y) \in D$. We need to prove that
$T$ is well defined. If $y \in \mbox{Fix}(\varphi)$, this follows
from Lemma~\ref{ref}. Take $y \in M \menos \mbox{Fix}(\varphi)$
and suppose that there exist $\tilde{g} \ne g$ such that
$\varphi ^g(y)=x \in D$ e $\varphi ^{\tilde{g}}(y) = \tilde{x}\in D$. In
this case, there exists a $l \in G$ such that $\varphi ^l (x) =
\tilde{x}$. We have that $\varphi ^{l^{-1}} \circ \varphi ^{\tilde{g}}
\circ \varphi ^{g^{-1}} (x)=x$, that is, $l^{-1}* \tilde{g}* g^{-1} =h
\in G_{x, \varphi}.$ Since  $x \in D \cap \varphi ^{h^{-1}}(U)$, so $h
\in G_{F(x),\psi}$. Therefore
\begin{eqnarray*}\psi^{\tilde{g}^{-1}} \circ F \circ \varphi ^{\tilde{g}}(y) &=& \psi^{g^{-1}} \circ \psi^{h^{-1}} \circ \psi^{l^{-1}}
\circ F \circ \varphi ^l \circ \varphi ^h \circ \varphi ^g (y) \\ &=&
\psi^{g^{-1}} \circ \psi^{h^{-1}} \circ \psi^{l^{-1}} \circ F
\circ \overbrace{\varphi ^l \circ \underbrace{\varphi ^g (y)}_{\in
D}}^{\in D\subset U} \\ &=& \psi^{g^{-1}} \circ \psi^{h^{-1}}
\circ F \circ  \varphi ^g (y) \\ &=& \psi^{g^{-1}} \circ F \circ
\varphi ^g  (y).
\end{eqnarray*}
This implies that $T$ is well defined and that is holomorphic,
since $U$ is open.  Given $y \in \Sat(D,\varphi)$, there
exist $h \in G$ and $x \in D$ such that $\varphi ^h(y)=x$. If $z \in
G$, then
\begin{eqnarray*}T \circ \varphi ^z(y) = T \circ \underbrace{\varphi ^z \circ \varphi ^{h^{-1}}(x)}_{\in \Sat(D,\varphi)} &=&
 \psi^{g^{-1}} \circ F \circ \underbrace{\varphi ^{g*z*h^{-1}}(x)}_{ \in D} \\ &=& \psi^z \circ \underbrace{\psi^{h^{-1}} \circ F \circ \varphi ^h}_{T}(y).
\end{eqnarray*}
By Lema $\ref{ref}$, $\psi \simeq_{qp} \varphi$ by $F^{-1}_E: V \to
U$. Thus we can construct analogously the inverse of $T$,
 and the proof is complete.
\end{proof}

%%%%%%%%%%%%%%%%%%%%%%%%%%%%%%%%%%%%%%%%%%
%%%%%%%%%%%%%%%%%%%%%%%%%%%%%%%%%%%%%%%%%%
%%%%%%%%%%%%%%%%%%%%%%%%%%%%%%%%%%%%%%%%%%

\begin{obs}\label{dic}\rm{Given a holomorphic  action $(\varphi,\C^*,M,p)$, consider the action
\begin{eqnarray*} \phi: \C \times M &\to& M \\  (z,x) &\to& \varphi (\exp(2 \pi \sqrt{-1}z), x).
\end{eqnarray*} Let $\xi$ be a local chart at $p$ and let $(\psi_p,\C^*,\C^n,0)$ be the action by linear transformations defined by the derivative of $\varphi$ at $p$, constructed from $\xi$. Let $X$ be the complete holomorphic vector field defined by $\varphi$. If we identify $\tilde{X}(x)=\sum_{j=1}^n a_j(x)\frac{\partial}{\partial z_j}$ with the map
$\tilde{X}(x)=(a_1(x), \dots, a_n(x))$, we have that $\psi_p$ is the flow of the linear vector field $Y(x)=D \tilde{X}(p) \cdot x$. Since the isotropy of $\psi_p$ at each point contains $\Z$, a Jordan canonical form for $D\tilde{X}(p)$ has no nilpotent part. Therefore, there is a linear isomorphism $A$ such that
$$A^{-1} \cdot \psi_p^z \cdot A(x) = \left(\exp(\lambda_1 2 \pi \sqrt{-1} z) x_1, \dots , \exp(\lambda_n 2 \pi \sqrt{-1} z) x_n\right),$$ where $\lambda_1,\dots , \lambda_n \in \Z$. We say that $p$ is a fixed point {\it of the dicritical type} if all  $\lambda_i$ are nonzero   with a same signal. This does not contradict the notion of dicriticity introduced above, about having all orbits near $p$ being contained in separatrices of the corresponding foliation $\fa_\vr$. }
\end{obs}

%%%%%%%%%%%%%%%%%%%%%%%%%%%%%%%%%%%%%%%%
%%%%%%%%%%%%%%%%%%%%%%%%%%%%%%%%%%%%%%%%
%%%%%%%%%%%%%%%%%%%%%%%%%%%%%%%%%%%%%%%%
\begin{proof}[Proof of Theorem~B] Take $\varphi$, $\psi_p$ and $\xi$ as in Remark $\ref{dic}$. We know that (up to linear conjugation)
\begin{equation}\label{flu}\psi_p (z, x) = \left(\exp(\lambda_1 2 \pi \sqrt{-1} z) x_1, \dots , \exp(\lambda_n 2 \pi \sqrt{-1} z) x_n\right),\end{equation}
where $x=(x_1,\dots,x_n)$, $\lambda_1,\dots , \lambda_n \in \Z$. Let $v = a - b \sqrt{-1} \in \C$.
The $j$-th coordinate of $\psi_p(z, x)$ is
\begin{equation}\label{coo}\exp\left(\lambda_j 2 \pi b \right) \cdot \exp\left(\lambda_j 2 \pi \sqrt{-1} a \right)x_j.\end{equation}
Since $p$ is of the dicritical type, there are no $\lambda_i$'s with different signals. Suppose that neither $\lambda_j$ is negative (the other case is identical). If $E$ is a open ball centered at origin, then
\begin{itemize}\item $\psi^z_p(E) \supset E$ if $b >0$,
               \item $\psi^z_p(E) \subset E$ if $b<0$ and
               \item$\psi^z_p(E)=E$ if $b =0$.
\end{itemize}
By Lemma $\ref{lli}$, there is a biholomorphism $F\colon  U \to V$ such that $\varphi \sim_{p0} \psi_p$ by $F$. Thus, $\psi_p \sim_{0p} \varphi$ by $F^{-1}: V \to U$. We can suppose that $V$ is an open ball. If we take an open ball $E \subset \subset V$ centered at origin, then follows from Lemma $\ref{comp}$ that $\psi_p \simeq_{0p} \varphi$ by $F^{-1}_E: V \to U$. So the desired result follows from Lemma $\ref{sat}$.
\end{proof}
%%%%%%%%%%%%%%%%%%%%%%%%%%%%%%%%%%%%%%%%
%%%%%%%%%%%%%%%%%%%%%%%%%%%%%%%%%%%%%%%%
%%%%%%%%%%%%%%%%%%%%%%%%%%%%%%%%%%%%%%%%

\section{Actions of $\bc^*$ on Stein manifolds}

Our above results lead to natural variants of results proven for holomorphic actions on Stein manifolds. For instance, joining the results in \cite{Scardua} and \cite{Camacho-Scardua} we can state the following
global linearization theorem:

\begin{Theorem}[\cite{CamachoScardua2}, Theorem 1.1]
\label{Theorem:2} Let $X$ be a complete holomorphic vector field
with isolated singularities on a Stein manifold $M$ of dimension
$n \ge 2$. Assume that $X$ has isolated singularities and some
dicritical singularity with first jet of the form
$X_{(\la_1,\dots,\la_n)}= \sum\limits_{j=1}^n
\la_jz_j\,\dfrac{\po}{\po z_j}$\,, where $\la_j \in \mathbb
Q_+$\,, $\forall\,j$. If $\sing X$ is finite and
$\overset{\,\vee}{H}^{\!}{^2}(M,\bz) = 0$ then $X$ is
holomorphically conjugate to $X_{(\la_1,\dots,\la_n)}$\,. In
particular, $M$ is biholomorphic to $\bc^n$.
\end{Theorem}

Applying our linearization results to this theorem we obtain:

\begin{theorem}\label{Theorem:dicriticalStein} Let $\vr$ be a $\bc^*$-action with isolated singularities on a Stein manifold $M$ of
dimension $n \ge 2$. Assume that $\vr$ has an isolated
dicritical singularity at $p\in M$. If  $\overset{\,\vee}{H}^{\!}{^2}(M,\bz) = 0$ then $\vr$ is analytically
linearizable, indeed, there is a biholomorphic map $F \colon M \to \bc^n$ which conjugates $\vr$
to the linear flow of  $X_{(\la_1,\dots,\la_n)}= \sum\limits_{j=1}^n
\la_jz_j\,\dfrac{\po}{\po z_j}$\,, where $\la_j \in \mathbb
Q_+$\,, $\forall\,j$. In
particular, $M$ is biholomorphic to $\bc^n$.
\end{theorem}

\begin{proof}
The $\bc^*$-action $\vr$ induces a holomorphic flow $\psi$ on $M$ by $ \psi_t(x)=
\vr_{e^t}(x), \, x \in M,  \, \,  t \in \bc$. The flow $\psi$ is $2\pi \sqrt{-1}$-periodic and induces a complete
holomorphic vector field $X$ on $M$ by $X(x) = \frac{\partial \psi_t (x)}{\partial t}\big|_{(t=0)}, \, \forall x \in M$.
By hypothesis $p\in M$ is a dicritical singularity for the action $\vr$ and by Theorem~B there is an analytic linearization for $\vr$ and therefore for $X$ in an invariant  neighborhood $W$ of $p$ in $M$. This implies that  $X$
has a first jet at $p$ of the form $X_{(\la_1,\dots,\la_n)}= a\cdot\sum\limits_{j=1}^n
\la_jz_j\,\dfrac{\po}{\po z_j}$\,, where $\la_j \in \mathbb
Q_+$\,, $\forall\,j$ for some $a\in \bc\setminus \{0\}$.
Applying then  Theorem~\ref{Theorem:2} above we conclude the proof.
\end{proof}

Next we obtain a sort of extension of Suzuki's linearization theorem in \cite{Suzuki2} to $\bc^*$-actions on
Stein surfaces:

\begin{theorem}
\label{Theorem:surface} Let $\vr$ be a $\bc^*$-action with isolated singularities on a Stein surface $N^2$ with $H^2(N^2,\mathbb Z)=0$ and $H_1 (N^2,\bc)=0$.
The we have the following possibilities:

\begin{itemize}
\item[{\rm(i)}] $\vr$ has a nondicritical singularity $p \in N^2$ and the action $\vr$ admits
a holomorphic first integral $F\colon V \to R$ {\rm(}$R=\mathbb D$
or $R=\bc${\rm)} of the form $F=f_1 ^{n_1} f_2 ^{n_2}$ where
$(f_j=0)$, j=1,2, are irreducible curves. The foliation $\fa_\vr$
is the pull-back of the linear foliation $\fa_1$ given on
$\bc^2_{(x,y)}$ by the vector field $X=x\frac{\partial}{\partial x}  -
y\frac{\partial}{\partial y}$ and corresponding to the action $\ro
\colon \bc^* \times \bc^2 \to \bc ^2$, $\ro (s,(x,y))=(s x,s^{-1}
y)$. The map $\mu:=(f_1^{n_1}, f_2^{n_2})$ is bijective as a map
between leaf spaces.

\item[{\rm(ii)}]  $\vr$ has a dicritical singularity, $\vr$ is globally linearizable and $N^2$ is biholomorphic to $\bc^2$.
\end{itemize}
\end{theorem}

\begin{proof} Because $N^2$ is Stein there is a meromorphic first integral $f\colon N^2 \dashrightarrow  \bc P(1)$ for $\vr$ on $N^2$ (cf. \cite{Suzuki1}). We can also assume that this first integral is primitive and is onto an open Riemann surface $R \in \{\bc, \mathbb D\}$ in the sense of \cite{Suzuki2} as explained in \cite{CamachoScardua2}. If  $\vr$ has no fixed point on $N^2$ then $f$ defines a holomorphic fibration what is not possible thanks to the hypothesis $H_1(N^2,\bc)=0$ (recall that the fibers are $\bc^*$ which is topologically a cylinder). Therefore $\vr$ must have some fixed point, i.e., some singularity for $\fa_\vr$. If the singularity is nondicritical then we apply \cite{CamachoScardua2}. If the singularity is dicritical then we apply Theorem~\ref{Theorem:dicriticalStein} above (essentially, \cite{Scardua}).

\end{proof}

\noindent{\bf Data availability statement}: 

\noindent  This manuscript is available at 
https://arxiv.org/abs/2408.09625.

\bibliographystyle{amsalpha}

\vglue.1in

\end{document}